\documentclass{article}
\usepackage{amsmath,amsfonts,amssymb,amsthm,hyperref,enumerate}

\usepackage[all]{xy}

\newtheorem{theorem}{Theorem}[section]
\newtheorem{proposition}[theorem]{Proposition}
\newtheorem{lemma}[theorem]{Lemma}
\newtheorem{example}[theorem]{Example}
\newtheorem{corollary}[theorem]{Corollary}
\newtheorem{condition}[theorem]{Condition}

\newtheorem{remark}[theorem]{Remark}

\newcommand{\op}{\textnormal{op}}

\newcommand{\C}{\mathbb{C}}
\newcommand{\X}{\mathbb{X}}
\newcommand{\V}{\mathcal{V}}

\newcommand{\Ke}{\textnormal{Ker}}
\newcommand{\ke}{\textnormal{ker}}
\newcommand{\Pt}{\mathbf{Pt}}

\newcommand{\Set}{\mathbf{Set}}

\DeclareMathOperator{\sq}{\fbox{\phantom{o}}}
\DeclareMathOperator{\sqz}{\fbox{\phantom{o}}^0}

\begin{document}
\title{A note on the Huq-commutativity of normal monomorphisms}
\author{James Richard Andrew Gray and Tamar Janelidze-Gray}
\maketitle
\abstract{
We give an alternative criteria for when a pair of Bourn-normal
monomorphisms Huq-commute in a unital category.
We use this to prove that in a unital
category, in which a morphism is a monomorphism if and only if its kernel is 
zero morphism, a pair of Bourn-normal monomorphisms with the same codomain
Huq-commute as soon as they have trivial pullback. As corollaries we
show that several facts known only in the protomodular context are in
fact true in more general contexts.
}
\section{Introduction}
It is well known and easy to prove that if $K$ and $L$ are normal subgroups
of a group $G$ and $K\cap L = 0$, then each element in $K$ commutes with each
element of $L$. This fact has several known generalizations to categories. 

An immediate generalization is obtained in the context where there is a
suitable notion of a commutator $[-,-]$  defined for normal subobjects, (which is
commutative and) satisfying
the property that if $K,L$ are normal subobjects then $[K,L]\leq K$.
In this context, if $K$ and $L$ are normal then  $[K,L]\leq K\wedge L$.
Therefore if $K$ and $L$ are trivial it immediately follows that $K\wedge L$
is trivial, which implies that $K$ and $L$ commute.
This is the case for the Huq commutator in a normal unital category.

An alternative generalization was obtained by D. Bourn
(Theorem 11 \cite{BOURN:2000a}) in the context of pointed
protomodular category \cite{BOURN:1991} (also introduced by D. Bourn):
he proved that if the meet of $k$
and $l$ is $0$, and
if $k$ and $l$ are Bourn-normal with the same codomain, then $k$ and $l$
Huq-commute \cite{HUQ:1968}. Recall that in a pointed finitely complete 
category,
a Bourn-normal monomorphism is essentially the zero class of an internal
equivalence relation.

We show (Corollary \ref{corollary:_meet_trivial_commutes}) that this latter
fact is true in the wider context of a unital category
\cite{BORCEUX_BOURN:2004} (introduced by F. Borceux and D. Bourn) satisfying
Condition \ref{zero_det}, which simply requires a morphism to be a monomorphism
as soon as it's kernel is zero. This context
is sufficiently wide so that it includes every normal unital category 
which implies that the former result also becomes a special case.
In doing so we
produce an alternative criteria (Theorem
\ref{theorem:char_of_huq_commutes_for_bourn_normal})
for when a pair of Bourn-normal monomorphisms commute in a unital category,
which closely resembles Proposition 2.6.13 of \cite{BORCEUX_BOURN:2004}.

We briefly study Condition \ref{zero_det}, and in particular: (i) we explain that it is a 
special case of a known condition (see Remark
\ref{remark:zero_dat_known}) and that it
together with regularity is easily equivalent to normality; (ii) we give
examples of categories satisfying it as well as our other conditions
(some of which are not normal categories); (iii) we characterize it in terms of
the fibration of points (Proposition \ref{characterization_of_zero_det}).
Using in part this characterization, we show that in a pointed
Mal'tsev category \cite{CARBONI_LAMBECK_PEDICCHIO:1991}
satisfying Condition \ref{zero_det}, the join of Bourn-normal
monomorphisms, with the same codomain and trivial meet, exists and is
Bourn-normal.
In addition, we show
that the characterization of abelian objects, via the normality of their
diagonal in the product, lifts from pointed protomodular categories to 
strongly unital categories \cite{BORCEUX_BOURN:2004} satisfying Condition \ref{zero_det}.

\section{Preliminaries}
In this section we recall the necessary definitions and preliminary facts, and introduce the
notation we will use. 

For a pointed category $\C$ we write $0$ for the
zero object as well as for each zero morphism between each pair of objects.
For objects $X$ and $Y$ we will often write $\pi_1 :X\times Y \to X$ and
$\pi_2 :X\times Y\to Y$
for the first and second product projections (when they exists), and for
morphisms $f:W\to X$ and $g:W\to Y$ we will write
$\langle f,g\rangle:W\to X\times Y$ for the unique morphism with
$\pi_1\langle f,g\rangle = f$ and $\pi_2\langle f,g\rangle=g$.
Recall that a category $\C$ is unital if $\C$ it is pointed, finitely complete,
and for each pair of objects $X$ and $Y$ the unique morphisms
$\langle 1,0\rangle: X\to X\times Y$ and $\langle 0,1\rangle : Y\to X\times Y$ are
jointly strongly epimorphic.

A pair of morphisms $f:X\to A$ and $g:Y\to A$ in a unital category $\C$
are said to Huq-commute, if
there exists a unique morphism $\varphi : X\times Y\to A$ making the diagram
\begin{equation}\label{}
\vcenter{
\xymatrix{
X
\ar[r]^-{\langle 1,0\rangle}
\ar[dr]_{f}
&
X\times Y
\ar[d]^{\varphi}
&
Y
\ar[l]_-{\langle 0,1\rangle}
\ar[dl]^{g}
\\
&
A
&
}
}
\end{equation}
commute. The morphism $\varphi$ is called the cooperator of $f$ and $g$. A
morphism $f:X\to A$ is called a central monomorphism if it is a monomorphism
and it Huq-commutes with $1_A$.

We will also need the following lemmas (see e.g \cite{BORCEUX_BOURN:2004} and the references there):
\begin{lemma}
\label{lemma:mono_cancels}
For $u:X'\to X$, $v:Y'\to Y$, $f:X\to A$ and $g:Y\to A$ morphisms in $\C$ and
$m:A\to B$ and monomorphism.
\begin{enumerate}[(i)]
\item The morphisms $f$ and $g$ Huq-commute if and only if the morphisms $g$
and $f$ Huq-commute;
\item  The morphisms $mf$ and $mg$ Huq-commute if and only if the morphisms
$f$ and $g$ Huq-commute;
\item  If the morphisms $f$ and $g$ Huq-commute, then so do the morphisms
$fu$ and $gv$.
\end{enumerate}
\end{lemma}
\begin{lemma}
\label{lemma:product_decomposes}
For $f:X\to A$, $g:Y\to A$, $f':X'\to A'$ and $g':Y'\to A'$ in $\C$,
the morphisms $f\times f'$ and $g\times g'$ Huq-commute if and only if
both the morphisms $f$ and $g$, and the morphisms $f'$ and $g'$ Huq-commute.
\end{lemma}
\begin{lemma}
\label{lemma:product_decomposes_2}
For $f:X\to A$, $g:Y\to A$, $f':X\to A'$ and $g':Y\to A'$ in $\C$,
the morphisms $\langle f, f'\rangle$ and $\langle g, g'\rangle$ Huq-commute if and only if
$f$ and $g$ Huq-commute, and $f'$ and $g'$ Huq-commute.
\end{lemma}

\section{The results} Throughout this section we assume that $\C$ is a unital category.
Let $k:X\to A$ and $l:Y\to A$ be monomorphisms, let
$r_1,r_2:R\to A$ and $s_1,s_2:S\to A$ be equivalence relations
and let $\kappa$ and $\lambda$ be morphisms such that the diagrams
\begin{equation}\label{Def k and l}\vcenter{
\xymatrix{
X
\ar[r]^{\kappa}
\ar[d]_{k}
&
R
\ar[d]^{\langle r_1,r_2\rangle}
&
Y
\ar[r]^{\lambda}
\ar[d]_{l}
&
S
\ar[d]^{\langle s_1,s_2\rangle}
\\
A
\ar[r]_-{\langle 1,0\rangle}
&
A\times A
&
A
\ar[r]_-{\langle 0,1\rangle}
&
A\times A
} }
\end{equation}
are pullbacks. Note that in pointed context this amounts to saying $k$ and $l$ are
Bourn-normal. In particular, this includes the case when $k$ and $l$ are the
kernels of some morphisms $f$ and $g$: in this case, $r_1,r_2$ and $s_1,s_2$
can be constructed as the kernel pairs of $f$ and $g$ respectively, and $\kappa$
and $\lambda$ are the unique morphisms with $r_1 \kappa = k$, $r_2\kappa=0$,
$s_1\lambda = 0$, and $s_2\lambda=0$.

We will need the relation $R\sqz S$, which is a pointed counter part to $R\sq S$
introduced by A. Carboni, M.C. Pedicchio and N. Pirovano in \cite{CARBONI_PEDICCHIO_PIRAVANO:1992}. In the context of pointed sets has
elements
\[
\{(x,a,y)\in X\times A\times Y\,|\,(k(x),a)\in S \text{ and } (a,l(y))\in R\}.
\]
Note that an element $(x,a,y)$ in $R\sqz S$ can, after identifying $k(x)$ and
$x$, and $l(y)$ and $y$, be displayed as follows
\[
\xymatrix@R=5ex@C=5ex{
x
\ar[r]^{R}
\ar[d]_{S}
&
0
\ar[d]^{S}
\\
a
\ar[r]_{R}
&
y.
}
\]
Categorically this relation can be built via the pullbacks 
\begin{equation}\vcenter{\label{Def Composite of relations}
\xymatrix{
&
&
P
\ar[dl]_{p_1}
\ar[dr]^{p_2}
&
&
\\
&
S\ar[dl]_{s_1}
\ar[dr]^{s_2}
&
&
R
\ar[dl]_{r_1}
\ar[dr]^{r_2}
&
\\
A
&
&
A
&
&
A
}}
\end{equation}
\begin{equation}\label{Def theta}\vcenter{
\xymatrix{
R\sqz S
\ar[r]^{\theta}
\ar[d]_{\psi}
&
X\times Y
\ar[d]^{k\times l}
\\
P
\ar[r]_-{\langle s_1p_1,r_2p_2\rangle}
&
A\times A
}}
\end{equation}
or directly as the limit of the outer arrows of what is easily seen to be a limiting cone
\begin{equation}
\label{limit_defining_box_zero}
\vcenter{
\xymatrix@C=6ex@R=7ex{
A
&
X
\ar[l]_{k}
&
\\
S
\ar[u]^{s_1}
\ar[d]_{s_2}
&
R\sqz S
\ar[u]^{\pi_1\theta}
\ar[l]_{p_1\psi}
\ar[r]^{\pi_2\theta}
\ar[d]^{p_2\psi}
&
Y
\ar[d]^{l}
\\
A
&
R
\ar[l]_{r_1}
\ar[r]^{r_2}
&
A.
}
}
\end{equation}
Let $\alpha:X\to R\sqz S$ and $\beta: Y\to R\sqz S$ be the unique cone morphisms
induced by the cones
\begin{equation}\label{Def alpha and beta}\vcenter{
\xymatrix@C=9ex@R=7ex{
A
&
X
\ar[l]_{k}
&
\\
S
\ar[u]^{s_1}
\ar[d]_{s_2}
&
X
\ar[u]^{1_X}
\ar[l]_{e_S k}
\ar[r]^{0}
\ar[d]^{\kappa}
&
Y
\ar[d]^{l}
\\
A
&
R
\ar[l]_{r_1}
\ar[r]^{r_2}
&
A
}
\xymatrix@C=9ex@R=7ex{
A
&
X
\ar[l]_{k}
&
\\
S
\ar[u]^{s_1}
\ar[d]_{s_2}
&
Y
\ar[u]^{0}
\ar[l]_{\kappa}
\ar[r]^{1_Y}
\ar[d]^{e_R l}
&
Y
\ar[d]^{l}
\\
A
&
R
\ar[l]_{r_1}
\ar[r]^{r_2}
&
A.
} }
\end{equation}

Note that, in particular, it follows that $\alpha$ and $\beta$ are morphisms making the two triangles in the diagram
\begin{equation*}
\label{theta_def}
\vcenter{
\xymatrix{
	& R\sqz S \ar[d]^{\theta}\\
X \ar[r]_-{\langle 1,0\rangle} \ar[ru]^-{\alpha} & X\times Y & Y \ar[l]^-{\langle 0,1\rangle}	\ar[lu]_{\beta},
}
}
\end{equation*}
where $\theta$ is defined as in Diagram \ref{Def theta},
are commutative. Since $\C$ is a unital catgory, this means (see e.g. Theorem 1.2.12 of \cite{BORCEUX_BOURN:2004}):
\begin{proposition}
\label{proposition:theta_strong}
The morphism $\theta$ in \eqref{theta_def} is a strong epimorphism. \qed
\end{proposition}
Using in part the previous fact, we are now ready to state and prove our
alternative criteria for when a pair of Bourn-normal monomorphisms commute.
\begin{theorem}
\label{theorem:char_of_huq_commutes_for_bourn_normal}
The following conditions are equivalent:
\begin{enumerate}[(a)]
\item $k:X\to A$ and $l:Y\to A$, as defined in \eqref{Def k and l}, Huq-commute;
\item $\alpha:X\to R\sqz S$ and $\beta:Y\to R\sqz S$, as defined in \eqref{Def alpha and beta}, Huq-commute;
\item $\theta :R\sqz S\to X\times Y$ is a split epimorphism of cospans with domain
$(R\sqz S,\alpha,\beta)$ and $(X\times Y,\langle 1,0\rangle,\langle 0,1\rangle)$.
\end{enumerate}
\end{theorem}
\begin{proof}
Let $m : R\sqz S \to X\times A\times Y$ be the morphism defined by
$m=\langle \pi_1\theta,s_2p_1\psi,\pi_2 \theta \rangle$. An easy calculation
shows that $m$ is a monomorphism. Noting that $m\alpha = \langle 1,k,0\rangle$
and $m\beta = \langle 0,l,1\rangle$, it follows from
Lemma
\ref{lemma:mono_cancels}
that  $\alpha$ and $\beta$ Huq-commute
if and only if $\langle 1,k,0\rangle$ and $\langle 0,l,1\rangle$ Huq-commute.
However, by Lemma
\ref{lemma:product_decomposes_2}
this latter condition is equivalent to requiring $k$ and $l$
to Huq-commute. This proves $(a)\Leftrightarrow (b)$.
To prove that $(b)\Rightarrow (c)$ we note that (b)
is equivalent to requiring that there is a morphism 
$\sigma : X\times Y \to R\sqz S$ making
the upper part of the diagram
\begin{equation}
\label{diag:char_of_huq_commutes_for_bourn_normal}
\vcenter{
\xymatrix{
 &
X\times Y
\ar[d]^{\sigma}
&
\\
X
\ar[ur]^{\langle 1,0\rangle}
\ar[r]^-{\alpha}
\ar[dr]_{\langle 1,0\rangle}
&
R\sqz S
\ar[d]^{\theta}
&
Y
\ar[ul]_-{\langle 0,1\rangle}
\ar[l]_-{\beta}
\ar[dl]^-{\langle 0,1\rangle}
\\
&
X\times Y
&
}
}
\end{equation}
commute. However, since $\langle 1,0\rangle$ and $\langle 0,1\rangle$ are jointly
epimorphic any such morphism must satisfy $\theta \sigma = 1_{X\times Y}$ and
so (c) holds. The converse is immediate, since (c) implies that there is a
morphism $\sigma$ making the upper part of
\eqref{diag:char_of_huq_commutes_for_bourn_normal} commute, and as mentioned
(b) is equivalent to the existence of such a morphism.

\end{proof}
\begin{lemma}
\label{lemma:kernel_of_theta}
The objects $\Ke(\theta)$ and $X\times_A Y$, where $X\times_A Y$ is the pullback of $k:X\to A$ and $l:Y\to A$, are isomorphic.
\end{lemma}
\begin{proof}
Note that since \eqref{Def theta} is a pullback, it follows that $\Ke(\theta)\cong \Ke(\langle s_1p_1,r_2p_2\rangle)$.
Now consider the diagram
\begin{equation}
\label{diag:proof_kernel_of_theta}
\vcenter{
\xymatrix@!@C=-0.75ex@R=0ex{
&
&
&
\Ke(\theta)
\ar[dl]_{v}
\ar[dr]^{u}
\ar@{}[dd]|{(*)}
\\
&
&
\Ke(s_1p_1)
\ar[dl]_{j}
\ar[dr]_(0.6){\ke(s_1p_1)}
&
&
\Ke(r_2p_2)
\ar[dl]^(0.6){\ke(r_2p_2)}
\ar[dr]^{i}
&
&
\\
&
Y
\ar[dl]
\ar[dr]^{\lambda}
&
&
P
\ar[dl]_{p_1}
\ar[dr]^{p_2}
&
&
X
\ar[dl]_{\kappa}
\ar[dr]
&
\\
0
\ar[dr]
&
&
S\ar[dl]_{s_1}
\ar[dr]^{s_2}
&
&
R
\ar[dl]_{r_1}
\ar[dr]^{r_2}
&
&
0
\ar[dl]
\\
&
A
&
&
A
&
&
A
&
}}
\end{equation}
consisting of the diagram \eqref{Def Composite of relations} and
in which:
\begin{description}
\item{-} $i$ and $j$ are the unique morphism such that $\lambda j =p_1 \ke(s_1p_1)$ and 
$\kappa i = p_2 \ke(r_2p_2)$;
\item{-} $u$ and $v$ are the unique morphisms making $(*)$ in the diagram above, commute.
\end{description}
Since each diamond in \eqref{diag:proof_kernel_of_theta} is a pullback and $r_1\kappa = k$ and $s_2\lambda =l$,
it follows that the diagram
\[
\xymatrix{
\Ke(\theta)
\ar[d]_{iu}
\ar[r]^{jv}
&
Y
\ar[d]^{l}
\\
X
\ar[r]_{k} & A
}
\]
is a also a pullback, and therefore, $\Ke(\theta)\cong X\times_A Y$ as desired.
\end{proof}
Let $\X$ be a pointed category. Consider the condition:
\begin{condition}\label{zero_det}
A morphism $f:A\to B$ in $\X$
is a monomorphism if and only if the kernel of $f$ is $0$.
\end{condition}
\begin{remark}\label{remark:zero_dat_known}
Note that a pointed category $\X$ satisfies Condition \ref{zero_det}
if and only if
each reflexive relation in $\X$ satisfies what was called Condition ($*\pi_0)$
in
\cite{GRAN_JANELIDZE:2014}, with respect to the ideal of zero morphisms.
\end{remark}
Recall that a regular category
\cite{BARR:1971}
is normal
\cite{JANELIDZE_Z:2010}
if and only if
every regular epimorphism is a normal epimorphism. 
The following proposition follows from Corollary 2.3
of
\cite{GRAN_JANELIDZE:2014}, however
we give a direct proof in order to avoid
introducing notation and terminology that would not
otherwise be needed in this paper.
\begin{proposition}
A regular category $\X$ with cokernels is normal if and only if it
satisfies Condition \ref{zero_det}.
\end{proposition}
\begin{proof}
It is immediate that a normal category satisfies Condition \ref{zero_det}. It remains
to prove the converse. Suppose $f:A\to B$ is a regular epimorphism
and consider the diagram
\[
\xymatrix{
\Ke(f)
\ar[r]^{\ke(f)}
\ar@{-->}[d]_{u}
&
A
\ar[r]^{f}
\ar@{-->}[d]^{q}
&
B
\\
\Ke(r)
\ar[r]_{\ke(r)}
&
Q
\ar[ur]_{r}
}
\]
in which $q$ is cokernel of $\ker(f)$, $r$ is the unique morphism
with $rq=f$, and $u$ the unique morphism with $\ke(r)u=q\ke(f)$.
Since the left hand square is a pullback it follows that $u$ is
a regular epimorphism. Since $\ke(r)u= q\ke(f)=0$, it follows that
$\ke(r)=0$, and therefore $r$ is monomorphism. Since $r$ is also a regular epimorphism, the latter implies that $r$ is an isomorphism.
\end{proof}
Recall that for a category $\X$ and an object $B$ is $\X$, the category
$\Pt_{\X}(B)$ of points, in the sense of D. Bourn, has objects triples
$(A,\alpha,\beta)$, where $A$ is an object in $\X$,  and $\alpha : A\to B$
and $\beta : B\to A$ are morphisms in $\X$ such that $\alpha\beta=1_B$. A
morphism $f$ from $(A,\alpha,\beta)$ to $(A',\alpha',\beta')$ in $\Pt_{\X}(B)$
is a morphism $f:A\to A'$, such that $\alpha'f=\alpha$ and $f\beta=\beta'$.
Furthermore, a morphism
$p:E\to B$ in $\X$ determines a pullback functor $p^*:\Pt_{\X}(B)\to \Pt_{\X}(E)$
which sends $(A,\alpha,\beta)$ in $\Pt_{\X}(B)$ to
$(E\times_B A,\pi_1,\langle 1,\beta p\rangle)$ in $\Pt_{\X}(E)$, with objects and
morphism defined as in the following commutative diagram
\[
\xymatrix{
E
\ar@/^2ex/[drr]^{\beta p}
\ar@/_2ex/[ddr]_{1_E}
\ar[dr]|{\langle 1,\beta p\rangle}
&
&
\\
&
E\times_B A
\ar[r]^{\pi_2}
\ar[d]_{\pi_1} 
\ar@{}[dr]|{\boxed{1}}
&
A
\ar[d]^{\alpha}
\\
&
E
\ar[r]_{p}
&B
}
\]
in which $\boxed{1}$ is a pullback. When $\X$ is a pointed category, pullback
functors along morphisms of the form $0\to B$ are essentially the same
as kernel functors $\text{Ker}_{B}:\Pt_{\X}(B)\to \X$.
\begin{proposition}\label{characterization_of_zero_det}
For a pointed finitely complete category $\X$ the following
are equivalent:
\begin{enumerate}[(a)]
\item The category $\X$ satisfies Condition \ref{zero_det};
\item For each object $B$ in $\X$ the functor $\textnormal{Ker}_B$ reflects terminal objects;
\item For each object $B$ in $\X$ the functor $\textnormal{Ker}_B$ reflects monomorphisms;
\item For each object $B$ in $\X$ the category $\Pt_{\X}(B)$ satisfies Condition \ref{zero_det};
\item For each morphism $p:E\to B$ in $\X$ the functor $p^*:\Pt_{\X}(B)\to \Pt_{\X}(E)$ reflects terminal objects;
\item For each morphism $p:E\to B$ in $\X$ the functor $p^*:\Pt_{\X}(B)\to \Pt_{\X}(E)$ reflects monomorphisms.
\end{enumerate}
\end{proposition}
\begin{proof}
For a morphism $f:A\to B$ in $\X$, note that:
\begin{enumerate}[(i)]
\item $f:A\to B$ is a monomorphism 
if and only if in the pullback diagram
\[
\xymatrix{
A\times_B A
\ar[d]_{\pi_1}
\ar[r]^-{\pi_2}
&A
\ar[d]^{f}
\\
A
\ar[r]_{f}
&
B
}
\]
 $\pi_1$ is a isomorphism.
\item  The morphism $\pi_1$ is an isomorphism whenever $(A\times_B A,\pi_1,\langle 1,1\rangle)$
is a terminal object in $\Pt_{\X}(A)$;
\item  The kernel of $f$ is isomorphic
to the kernel of $\pi_1$.
\end{enumerate}
 Combining these observations we see that
(a)$\Leftrightarrow$(b).
For any functor $F$ between pointed categories which preserves terminal objects,
 since morphisms into the terminal object
are necessarily split epimorphisms, one easily shows that if $F$ reflects
monomorphisms, then it reflects terminal objects.
Therefore (f)$\Rightarrow$(e) and (c)$\Rightarrow$(b). 
Recalling that if a composite of functors $FG$ reflects
some property and $F$ preserves it, then $G$ reflects it, and
noting that kernel functors certainly preserve terminal objects,
one easily sees that (b)$\Rightarrow$(e) (just note that for each
morphism $p:E\to B$ the functor $\textnormal{Ker}_E\circ p^*$ is isomorphic to
$\textnormal{Ker}_B$). Since
each pullback functor between points along a morphism in a category of points
of  $\X$ is up to
isomorphism a pullback functor between points for $\X$ it follows that
(e)$\Rightarrow$(d). 
For a functor $F$ between 
pointed finitely complete categories
satisfying Condition 3.4, preserving limits and reflecting terminal objects,
if $F(f)$ is a monomorphism then
$F(\text{Ker}(f))\cong \text{Ker}(F(f)) \cong 0$ and hence
$\text{Ker}(f)\cong 0$ which forces $f$ to be a monomorphism.
This proves (e)$\Rightarrow$(f) since we already know that (e)$\Rightarrow$(d).
The proof is completed by noting that
trivially (f)$\Rightarrow$(c) and (d)$\Rightarrow$(a).
\end{proof}
\begin{proposition}
Let $\V$ be a (quasi)-variety of universal algebras considered as a category,
and let $\X$ be a category with finite limits.
If $\V$ satisfies Condition \ref{zero_det}, then $\V(\X)$ satiesfies
Condition \ref{zero_det}.
\end{proposition}
\begin{proof}
Since the Yoneda embedding $Y : \X \to \Set^{\X^{\op}}$ preserves and reflects limits
and $\V(\Set^{\X^{\op}})=\V^{\X^{\op}}$, taking internal $\V$ algebras we 
obtain a functor $\tilde Y : \V(\X) \to \V^{\X^{\op}}$
which preserves and reflects limits. The claim now follows by noting that
Condition \ref{zero_det} lifts to functor categories.
\end{proof}
\begin{example}
Recall that an implication algebra is a triple $(X,\cdot,1)$ where $X$
is a set,
$\cdot$ is a binary operation and $1$ is constant satisfying the
axioms:
$(xy)x=x$, $(xy)y=(yx)x$, $x(yz)=y(xz)$, $11=1$. H.\ P.\ Gumm and A.\ Ursini showed
in \cite{GUMM_URSINI:1984} that the variety of implication algebras form
an ideal determined variety of universal algebras which is not
congruence permutable. This means that the category of implication algebras
is ideal determined but not Mal'tsev \cite{JANELIDZE_MARKI_THOLEN_URSINI:2010}.
Since the two element boolean algebra $2=(2,\to,1)$ forms an implication
algebra and $\{(0,1),(1,0),(1,1)\}$ is a sub-algebra of $2\times 2$,
we see that it is not a unital category. However adding an independent
binary operation $*$ satisfying  $x*1=1*x=x$ will produce a unital
ideal determined category, and hence a strongly unital normal category.
We leave as open problems whether this latter variety is Mal'tsev or not
and if their exists a normal strongly unital variety which is not Mal'tsev.
On the other hand the previous proposition tells us that internal such
algebras in a category with finite limits always produce a category which is
strongly unital and satisfies Condtion \ref{zero_det}.
\end{example}
\begin{example}
It is easy to show that the quasi-variety $\mathcal{V}$ of universal algebras, with terms
$p(x,y)$ and $s(x,y)$ satisfying $p(x,0)=p(0,x)=x$, $s(x,0)=x$, $s(x,x)=0$,
and $s(x,y)=0 \Rightarrow x=y$, is a normal strongly unital category.
In fact, it turns out that this quasi-variety is
almost exact (i.e every regular epimorphism is an effective descent morphism)
and is not Mal'tsev. As before, by the previous proposition, we obtain that
internal such algebras in a finitely complete catgory will produce strongly
unital categories satisfying Condition \ref{zero_det}. In particular, if the
base category is the product of the category of sets with the quasi-variety
$\mathcal{W}$ of abelian groups satisfying $4x=0 \Rightarrow 2x=0$, then
resulting category
will on the one hand not be Mal'tsev since $\mathcal{V}$ is not, and on the
other hand not be regular (and hence not normal) since
$\mathcal{V}(\mathcal{W})=\mathcal{W}$ which is not regular.
\end{example}
\begin{corollary}
\label{corollary:_meet_trivial_commutes}
Let $k$ and $l$ be Bourn-normal monomorphisms in a unital
category $\C$ satisfying Condition \ref{zero_det}.
If $k$ and $l$ have trivial pullback, then $k$ and $l$ commute.
\end{corollary}
\begin{proof}
If $k$ and $l$ have trivial pullback, then by Lemma
\ref{lemma:kernel_of_theta}
the morphism
$\theta$ has trivial kernel and hence is a monomorphism. Moreover, since by
Proposition
\ref{proposition:theta_strong}
the morphism
$\theta$ is strong epimorphism it foolws that it is an isomorphism.
The claim now follows from Theorem
\ref{theorem:char_of_huq_commutes_for_bourn_normal}
.
\end{proof}
\begin{lemma}
\label{lemma:kernel_of_coperator}
Let $k:X\to A$ and $l:Y\to A$ be monomorphisms in
a strongly unital category $\C$
which commute, and let
$\varphi : X\times Y \to A$ be their cooperator. If 
$\langle u,v \rangle : W\to X\times Y$ is the kernel of $\varphi$,
then $u$ and $v$ are central monomorphisms and $ku=k(-v)$. 
\end{lemma}
\begin{proof}
Since each of the squares in the following diagram
\[
\xymatrix{
0
\ar[r]
\ar[d]
&
W
\ar[d]_{\langle u,v\rangle}
&
0
\ar[l]
\ar[d]
\\
X
\ar[r]_{\langle 1,0\rangle}
&
X\times Y
&
Y
\ar[l]^{\langle 0,1\rangle}
}
\]
are pullbacks, we see via Lemma \ref{lemma:product_decomposes}, that $u$ and
$v$ are central.
To complete the proof just note that
$\langle u,0\rangle = \langle u,v\rangle+\langle 0,-v\rangle$
and therefore
\begin{align*}
ku&=\varphi \langle u,0\rangle\\
&= \varphi (\langle u,v\rangle + \langle 0,-v\rangle)\\
&= \varphi \langle u,v\rangle + \varphi \langle 0,-v\rangle\\
&= \varphi \langle 0,-v\rangle\\
&=l(-v).
\end{align*}
\end{proof}
\begin{proposition}
\label{prop:copoerator_join}
Let $k$ and $l$ be Bourn-normal monomorphisms in a strongly unital
category $\C$ satisfying Condition \ref{zero_det}.
If $k$ and $l$ have trivial pullback, then $k$ and $l$ commute
and their cooperator $\varphi : X\times Y\to A$ is a monomorphism,
which is also their join.
\end{proposition}
\begin{proof}
By Corollary \ref{corollary:_meet_trivial_commutes} we know that $k$ and $l$ commute.
It remains
to show that their cooperator is a monomorphism and is their join. The
first point follows from Condition \ref{zero_det} since
Lemma \ref{lemma:kernel_of_coperator} implies the kernel of $\varphi$ is zero. The final point follows
immediately from the fact that $\langle 1,0\rangle$ and $\langle 0,1\rangle$
are jointly strongly epimorphic.
\end{proof}
Recall that in the Mal'tsev context an equivalence relation $r_1,r_2:R\to A$ is
essentially the same thing as a monomorphism
\[
\xymatrix{
R
\ar@<0.5ex>[dr]^{r_1}
\ar[rr]^-{\langle r_1,r_2\rangle}
&
&
A\times A
\ar@<0.5ex>[dl]^{\pi_1}
\\
&
A
\ar@<0.5ex>[ul]^{e}
\ar@<0.5ex>[ur]^{\langle 1,1\rangle}
&
}
\]
in the category $\Pt(A)$. Moreover such a monomorphism $\langle r_1,r_2 \rangle$ is
necessarily Bourn-normal. To see why, consider the pullback diagram
\[
\xymatrix{
R\times_A R
\ar[d]_{p_1}
\ar[r]^{p_2}
&
R
\ar[d]^{r_1}
\\
R
\ar[r]_{r_1}
&
A.
}
\]
It follows that
$\langle r_1p_1,r_2p_1\rangle,\langle r_1p_1,r_2p_2\rangle:
(R\times_A R,r_1p_1,\langle e,e\rangle)
\to
(A\times A,\pi_1,\langle 1,1\rangle)$
(where $e$ is the splitting of $r_1$ and $r_2$) is an equivalence relation and the diagrams
\[
\xymatrix{
A\times (A\times A)
\ar[d]_{1\times \pi_1}
\ar[r]^-{1\times \pi_2}
&
A\times A\ar[d]^{\pi_1}\\
A\times A\ar[r]_{\pi_1} & A
}
\xymatrix@C=10ex{
R
\ar[d]_{\langle r_1,r_2\rangle}
\ar[r]^-{\langle 1,r_1 e\rangle}
&
R\times_A R
\ar[d]^{\langle r_1\pi_1,\langle r_2p_1,r_2p_2\rangle\rangle}
\\
A\times A
\ar[r]_-{1\times \langle p_2,p_1\rangle}
&
A\times (A\times A)
}
\]
are pullbacks.
\begin{theorem}
Let $k$ and $l$ be Bourn-normal monomorphisms in a Mal'tsev
category $\C$ satisfying Condition \ref{zero_det}.
If $k$ and $l$ have trivial pullback, then $k$ and $l$ commute
and their cooperator $\varphi : X\times Y\to A$ is a Bourn-normal
monomorphism which is also their join.
\end{theorem}
\begin{proof}
By Proposition \ref{characterization_of_zero_det} we see that $X\wedge Y=0$
implies $R\wedge S=0$ when considered as subobjects
of $(A\times A,\pi_1,\langle 1,1\rangle)$ in the category of points
over $A$.
It now follows from 
Proposition \ref{prop:copoerator_join} that
$R\times S$ (in $\Pt(A)$) is a subobject of $(A\times A,\pi_1,\langle 1,1\rangle)$, and
hence is an equivalence relation with zero class the cooperator of
$k$ and $l$.
\end{proof}
\begin{corollary}
Let $\C$ be strongly unital category satisfying Condition \ref{zero_det}.
For an object $X$ in $\C$ the following conditions are equivalent:
\begin{enumerate}[(a)]
\item $X$ is abelian;
\item $\langle 1,1\rangle : X \to X\times X$ is a normal monomorphism;
\item $\langle 1,1\rangle : X\to X\times X$ is a Bourn-normal monomorphisn.
\end{enumerate}
\end{corollary}
\begin{proof}
The implications (a) $\Rightarrow$ (b)$\Rightarrow$ (c) are immediate.
It is therefore sufficient to show that (a) follows from (c).
Suppose $X$ is object in $\C$ and $\langle 1,1\rangle$ is Bourn-normal.
Since $\langle 1,0\rangle$ and $\langle 1,1\rangle$ have trivial pullback,
it follows that they commute and hence we obtain a morphismi $\psi$ making
the diagram
\[
\xymatrix{
X
\ar[r]^-{\langle 1,0\rangle}
\ar@{=}[d]
&
X\times X
\ar[d]^{\psi}
&
X
\ar[l]_-{\langle 0,1\rangle}
\ar@{=}[d]
\\
X
\ar[r]^-{\langle 1,0\rangle}
&
X\times X
&
X
\ar[l]_-{\langle 1,1\rangle}
}
\]
commute. The claim now follows from Corollary 1.8.20 of
\cite{BORCEUX_BOURN:2004}, since $\pi_1\psi$ is a cooperator for
$1_X$ and $1_X$.
\end{proof}
\begin{remark}
Given that abelianess is a property in a subtractive category, and
abelianization is obtained by forming the cokernel of the diagonal
in a regular subtractive category (provided the cokernel exists)
\cite{BOURN_JANELIDZE_Z:2016}, one
expects that the above corollary is true in a wider context.
\end{remark}

\end{document}